\documentclass[12pt]{amsart}

\usepackage{amsmath,amssymb,amscd,amsxtra}
\usepackage{graphics,graphicx}
\usepackage{latexsym}

\newcommand{\N}{\mathbb{N}}
\newcommand{\R}{\mathbb{R}}
\newcommand{\C}{\mathbb{C}}
\newcommand{\Z}{\mathbb{Z}}
\newcommand{\Q}{\mathbb{Q}}

\renewcommand{\S}{\mathcal{S}}

\newcommand{\ip}[2]{\langle#1,#2\rangle}

\newtheorem{thm}{Theorem}
\newtheorem{lemma}{Lemma}
\newtheorem{prop}{Proposition}
\newtheorem{cor}{Corollary}
\theoremstyle{remark}
\newtheorem{rem}{Remark}
\theoremstyle{definition}
\newtheorem{deft}{Definition}
\newtheorem{example}{Example}
\theoremstyle{conjecture}
\newtheorem{conjecture}{Conjecture}
\theoremstyle{claim}

\begin{document}

\title{Extension and restriction principles for the HRT conjecture}

\author{Kasso A.~Okoudjou}

\address{Kasso A.~Okoudjou\\
Department of Mathematics $\&$ Norbert Wiener Center\\
University of Maryland\\
College Park, MD, 20742 USA}

\email{kasso@math.umd.edu}
\thanks{This work  was partially supported by a grant from the Simons Foundation $\# 319197$, and ARO grant W911NF1610008. Part of this material is based upon work supported by the National Science Foundation under Grant No.~DMS-1440140 while the author was in residence at the Mathematical Sciences Research Institute in Berkeley, California, during the Spring 2017 semester.}


\date{\today}


\begin{abstract}The HRT (Heil-Ramanathan-Topiwala)  conjecture asks whether a finite collection of time-frequency shifts of a  non-zero square integrable function on $\R$  is linearly independent. This longstanding conjecture remains largely open even in the case when the function is assumed to be smooth. Nonetheless, the conjecture has been proved for some special families of functions and/or special sets of points. The main contribution of this paper is an inductive approach to investigate the HRT conjecture  based on the following. Suppose that the HRT is true for a given set of $N$ points and a given function. We identify  the set of all  new points such that the conjecture remains true for the same function and the set of $N+1$ points obtained  by adding one of these new points to the original set. To achieve this we introduce a real-valued function whose global maximizers describe when the HRT is true. To motivate this new approach we re-derive a special case of the HRT for sets of $3$ points. Subsequently, we establish   new results for points in $(1,n)$ configurations, and for a family of symmetric $(2,3)$ configurations. Furthermore, we use these results and the refinements of other known ones  to prove that the HRT holds for certain families of $4$ points. 
\end{abstract}

\maketitle \pagestyle{myheadings} \thispagestyle{plain}
\markboth{K. A. OKOUDJOU}{EXTENSION PRINCIPLE FOR THE HRT}

\section{Introduction}\label{sec1}

 For $a, b \in \R$  and a function $g$ defined on $\R$, let $M_bf(x)=e^{2\pi i bx}f(x)$ and $T_af(x)=f(x-a)$  be  respectively  the modulation operator, and the translation operator.   Given a function $g\in L^{2}(\R)$
 and $\Lambda=\{(a_k, b_k)\}_{k=1}^{N} \subset \R^2$, we define $$\mathcal{G}(g, \Lambda)=\{e^{2\pi i b_k \cdot}g(\cdot - a_k)\}_{k=1}^{N}.$$ $\mathcal{G}(g, \Lambda)$ is called a (finite) Weyl-Heisenberg or Gabor system \cite{Groc2001}. The HRT conjecture \cite{Heil06, HRT96}, states that 

\begin{conjecture}\label{hrt}
Given any $0\neq g \in L^{2}(\R)$  and $\Lambda=\{(a_k, b_k)\}_{k=1}^{N} \subset \R^2$, $\mathcal{G}(g, \Lambda)$ is a linearly independent set in $L^2(\R)$.
\end{conjecture}

To date a definitive answer to the Conjecture has not been given even  when one assumes that the function $g$ is very smooth and decays fast, e.g., when $g \in S(\R)$, the space of Schwartz functions on $\R$. In particular, the following (sub-conjecture) is also open

\begin{conjecture}\label{hrt-s}
Given any $ g \in S(\R), g\neq 0$  and $\Lambda=\{(a_k, b_k)\}_{k=1}^{N} \subset \R^2$, $\mathcal{G}(g, \Lambda)$ is a linearly independent set in $L^2(\R)$.
\end{conjecture}
 While the statement of the problem seems simple, a variety of sophisticated tools such as the ergodic theorems, von Neumann algebra methods, number theory arguments, random Schr\"odinger operators,  harmonic analysis, operator theory, has  been used to prove the few known results.  Perhaps the lack of unifying theme in the proofs of the known results attests to  the difficulty of this problem.
 
  The HRT conjecture contains two fundamental data: the function $g \in L^2(\R)$ and the set of points $\Lambda=\{ (a_k, b_k)\}_{k=1}^N \subset \R^2$. Most of the known results either assume $g \in L^2$ and $\Lambda$ is restricted to some special family of points, or that $\Lambda$ is very general and restrictions are imposed on $g$.  We outline all the known results about the HRT conjecture of which we are aware and we refer to the surveys  \cite{Heil06, HeiSpe15} for more details.

\begin{prop}\label{known-results} The following statements hold.
\begin{enumerate}
\item[(i)] Conjecture~\ref{hrt} holds for any $\Lambda \subset \R^2$, when $g$ is compactly supported, or just supported within a half-interval $(-\infty, a], $ or $[a, \infty)$ \cite{HRT96}.
\item[(ii)] Conjecture~\ref{hrt} holds for any $\Lambda \subset \R^2$, when $g(x)=p(x)e^{-\pi x^2}$ where $p$ is a polynomial \cite{HRT96}.
\item[(iii)] Conjecture~\ref{hrt} holds for any $g \in L^{2}(\R)$, when $\Lambda$ is a finite set with  $\Lambda \subset A(\Z^2) + z$ where $A$ is a full rank $2\times 2$ matrix and $z\in \R^2$ \cite{Lin99}. In particular, Conjecture~\ref{hrt} holds when $\# \Lambda \leq 3$ for any $g \in L^2$ \cite{HRT96}.
\item[(iv)] Conjecture~\ref{hrt} holds for any $g \in L^2$, when $\# \Lambda =4 $ and two of the four points in $\Lambda$ lie on a line and the remaining two points lie on a second parallel line \cite{Dem10, DemZah12}.
\item[(v)] Conjecture~\ref{hrt-s} holds for any $g \in S(\R)$, when $\# \Lambda =4 $ and three  of the four points in $\Lambda$  lie on a line and the fourth point is off this line  \cite{Dem10}.
\item[(vi)] Conjecture~\ref{hrt} holds for any $\Lambda \subset \R^2$, when $\lim_{x \to \infty}|g(x)|e^{cx^{2}}=0$ for all $c>0$ \cite{BowSpee13}. 
\item[(vii)] Conjecture~\ref{hrt} holds for any $\Lambda \subset \R^2$, when $\lim_{x \to \infty}|g(x)|e^{cx\log x}=0$ for all $c>0$ \cite{BowSpee13}. 
\item[(viii)] Conjecture~\ref{hrt} holds when $g$ is ultimately positive, and  $\Lambda =\{(a_k, b_k)\}_{k=1}^N \subset \R^2$ is such that $\{b_k\}_{k=1}^N$ are independent over the rationals $\Q$ \cite{BeBo13}. 
\item[(ix)] Conjecture~\ref{hrt} holds for any $\# \Lambda =4 $, when $g$ is ultimately positive, and $g(x)$ and $g(-x)$ are ultimately decreasing \cite{BeBo13}.  
\item[(x)] Conjecture~\ref{hrt} holds for any $g \in L^{2}(\R)$, when $\Lambda$ consists of collinear points \cite{HRT96}.
\item[(xi)] Conjecture~\ref{hrt} holds for any $g \in L^{2}(\R)$, when $\Lambda$ consists of $N-1$ collinear and equi-spaced  points, with the last point located off this line \cite{HRT96}.
\end{enumerate}
\end{prop}

We note that there is some redundancy in Proposition~\ref{known-results} as part (vii) implies parts (i), (ii), and (vi). Nonetheless, we include all these results to give an historical perspective on the HRT conjecture.  In addition to these,  perturbation arguments \cite{HRT96} have  been used on either the function $g$ or the set $\Lambda$ to get related results. A spectral result related to the HRT  has been presented in \cite{RVBal08, BalKri10},  and estimates of frame bounds for Gabor systems related to the HRT conjecture have appeared in \cite{ChristenLind01, Groc14}. A connection between the HRT, the Bargmann-Fock space and the Segal-Bargmann transform was presented in \cite{Stroock_HRT}.  Other results concerning the HRT can be found in \cite{BowSpee10, DemZub13, Zioma}, and  for an overview of the status of the conjecture we refer to \cite{Heil06, HeiSpe15}. When $g \in L^{2}(\R^d)$, $d\geq 2$,  and $\Lambda \subset \R^{2d}$ not much is known about the conjecture, see \cite{BowSpee16}. We refer  to \cite{Ros08} for a related problem for pure translation systems, and to \cite{Kut02} for some generalizations of the conjecture.

A set $\Lambda$  of the form given in (iv) or (v) of Proposition~\ref{known-results}, is referred to as a $(2,2)$ configuration  and $(1,3)$ configuration, respectively. More generally, 
\begin{deft}
An $(n, m)$ configuration is a collection of $n+m$ distinct points in the plane, such that there exist $2$ distinct parallel lines such that one of them contains exactly $n$ of the points and the other one contains exactly $m$ of the points. 
\end{deft}

One of the goals of this paper is to present two different approaches to investigate the HRT conjecture. On the one hand, we prove  an \emph{extension principle} and use it to attack the HRT conjecture.  No such extension or other inductive methods related to the HRT have ever been proved. More specifically, knowing that the Conjecture holds for a given function $g \in L^{2}(\R)$ and a given set $\Lambda=\{(a_k, b_k)\}_{k=1}^{N} \subset \R^2$, we identify the set of all (new) points $(a, b) \in \R^2 \setminus \Lambda$ such that the conjecture remains true for the same function $g$ and the new set $\Lambda'=\Lambda \cup \{(a,b)\}$. On the other hand, we consider the related \emph{restriction principle} which asks the following question; knowing that the HRT is true for a specific set of $N+1$ points and $g$, can one establish the conjecture for a family of $N$ related points and the same function $g$? 
To answer these questions we introduce a real-valued function that is generated by the two data in the HRT conjecture, namely, the function $g$ and the set $\Lambda$. As we shall show, this function is derived from the Gramian of $\mathcal{G}(g, \Lambda)$ and is based on a fundamental time-frequency analysis tool: the short-time Fourier transform. Using this function along with  refinements of some of the techniques introduced by Demeter in \cite{Dem10} allow us to recover some known results and  establish new ones. In particular, the main contributions of this paper are: \newline
$\bullet$ a proof that  HRT conjecture holds for all $(1,3)$ configurations when $g$ is real-valued, \newline
$\bullet$ a proof that HRT holds for a family of symmetric $(2,3)$ configurations, \newline
$\bullet$ a proof that HRT holds for  a large family of $4$ points (not in $(1,3)$ nor $(2,2)$ configurations) and real-valued functions in $L^2(\R)$.

Furthermore, as a byproduct of our approach we obtain:\newline
$\bullet$  a new proof of HRT for collinear points, \newline
$\bullet$ a new proof of HRT for sets of $3$ unit-lattice points and real-valued functions.

The rest of the paper is organized as follows. In Section~\ref{sec2}, we introduce some of the technical tools needed to state our results. We then use Bochner's theorem to provide a new proof of the HRT conjecture for collinear points (Theorem~\ref{hrt-colinear}). We also motivate the extension principle by offering a new proof of the HRT conjecture for $3$ points on the unit lattice and real-valued functions  (Proposition~\ref{L2-3}). 
In Section~\ref{subsec3.1} we  introduce and collect the main properties of the extension function which is the basis of the extension principle we propose. Subsequently, we prove in Section~\ref{subsec3.2} that there exists  at most one (equivalence class of) $(1,n)$ configuration for which the HRT conjecture could fail whenever $n\geq 3$ (Theorem~\ref{1-nconfig}). Furthermore, when the generator is real-valued we show that   the HRT holds for all $(1,3)$ configurations (Theorem~\ref{1-3config}). In Section~\ref{subsec3.3} we introduce the restriction principle. For this case, we refine Demeter's ``conjugate trick'' argument to establish both Conjecture~\ref{hrt} (Theorem~\ref{3-2config}) and  Conjecture~\ref{hrt-s} (Theorem~\ref{smooth-3-2config}) for a family of symmetric $(2,3)$ configurations.  Subsequently, we apply the restriction principle to  prove Conjecture~\ref{hrt} for real-valued functions and a related family of $4$ points that are not  $(1,3)$ nor $(2,2)$ configurations (Corollary~\ref{4points-hrt}).

\section{Preliminaries and motivation}\label{sec2}
In this section, we collect some properties of the Short-Time Fourier Transform
(STFT)  as  well as some results concerning positive definite matrices, see Section~\ref{subsec2.1}. Using the Gramian of $\mathcal{G}(g, \Lambda)$ and Bochner's theorem we then give a new proof of the HRT conjecture for collinear points, see Section~\ref{subsec2.2}. Finally, in Section~\ref{subsec2.3} we revisit the HRT for $3$ points and provide a new proof of the validity of the conjecture in this case. This new proof serves as a motivation for the extension principle that we propose. The methodology we develop below is fundamentally based on the analysis of the Gramian of $\mathcal{G}(g, \Lambda)$. In particular, the notions of positive definiteness of functions and matrices constitute the overarching themes of this methodology.
\subsection{Preliminaries}\label{subsec2.1}

Let $f, g  \in L^{2}(\R)$. The Short-Time Fourier Transform
(STFT) of a function $f$ with respect to a window $g$ is
$$V_g f(x, y)=\int_{\R} f(t)\, \overline{g(t-x)} \, e^{-2\pi i y  t} \, dt.$$
It is easy to prove that $V_g f$ is 
a  bounded  uniformly continuous function on $\R^2$, and that $$\lim_{|x|, |y|\to \infty}V_gf(x, y)=0;$$see
 \cite{Groc2001}. We will also need the following the orthogonality and covariance properties of the STFT:
given $g_i, f_i, \in L^2(\R)$, $i=1, 2,$ we have 
\begin{equation}\label{ortho-Vgg}
 \ip{V_{g_{1}}f_{1}}{V_{g_2}f_2}=\ip{f_{1}}{f_2}\overline{\ip{g_1}{g_2}},
\end{equation} and 
\begin{equation}\label{stft-TF}
 V_{g}(T_{a}M_{b}f)(x, y)=e^{-2\pi i ay}V_gf(x-a, y-b),
\end{equation} see \cite[Lemma 3.1.3]{Groc2001}, and \cite[Theorem 3.2.1]{Groc2001}.

We will also use the following formula whenever it is well defined: 
\begin{equation}\label{FT-prod-stft}
\mathcal{F}_2(V_{g_{1}}f_1 \overline{V_{g_{2}}f_2})(\xi, \eta) = \big(V_{f_{2}}f_1 \overline{V_{g_{2}}g_1}\big)(-\eta, \xi),
\end{equation} where $\mathcal{F}_2$ denotes the two dimensional Fourier transform. We refer to \cite{GroZim01} for a proof of this statement. 

We need some facts about positive definite matrices. We refer to \cite[Section 7. 7]{HornJohnson85} for details. In particular, given $N \times N$ Hermitian matrices $A,$ and $B$, we write $A \succ B$ if $A-B$ is positive definite. We will repeatedly use the following theorem.

\begin{thm}\cite[Theorem 7.7.6]{HornJohnson85} \label{psd-matrix}
Let $E$ be  an Hermitian $N \times N$ matrix such that $$E=\begin{bmatrix}A & B \\ B^* & C\end{bmatrix},$$ where $A, C$ are square matrices. Then, $E$ is positive definite if and only if $A$ is positive definite and $C\succ B^*A^{-1}B.$ 
\end{thm}

In our setting,  $$E=(\ip{f_k}{f_\ell})_{k, \ell=1}^N$$will be the Gramian of a set of $N$ functions $\{f_k\}_{k=1}^N \subset L^{2}(\R)$. As a consequence, $E$ is automatically positive semidefinite, i.e.,  $E\succeq 0$. Furthermore, $A$ and $C$ will always be positive definite matrices. In particular, we will consider the case where $A$ is an $(N-1)\times (N-1)$ positive definite matrix, $C=1$, and $B=u$ will be a vector in $\C^{N-1}$. In this case we will make repeated use of the following corollary of Theorem~\ref{psd-matrix}.

\begin{cor}\label{version-psd} With the above notations the following assertions hold:
\begin{enumerate}
\item $E\succeq 0$ if and only if $\ip{A^{-1}u}{u}\leq 1.$ Furthermore, $E\succ 0$ if and only if $\ip{A^{-1}u}{u}< 1.$ Consequently, $E$ is singular if and only if $\ip{A^{-1}u}{u}= 1.$
\item $\det{E}=(1-\ip{A^{-1}u}{u})\det{A}.$
\end{enumerate}
\end{cor}

\begin{proof}
The proof is given in~\cite[Theorem 7.7.6]{HornJohnson85}. However, we outline it for the sake of completeness. 

We are given that $A$ is positive definite. Now assume that $E$ is positive semidefinite and let $X=-A^{-1}u$. Then 
$$\begin{bmatrix}I & 0\\X^{*} & 1 \end{bmatrix} \begin{bmatrix}A & u \\ u^{*} & 1\end{bmatrix}\begin{bmatrix}I & X \\ 0 & 1\end{bmatrix}=\begin{bmatrix}A &0 \\ 0 & 1-\ip{A^{-1}u}{u}\end{bmatrix}$$ is positive semidefinite. Thus $1\geq \ip{A^{-1}u}{u}$. And the converse is trivially seen. 

The last two parts easily follow as well. 
\end{proof}

\subsection{Revisiting the HRT for collinear points}\label{subsec2.2}
To the best of our knowledge no result on the HRT conjecture has been obtained through the  analysis of the Gramian of $\mathcal{G}(g, \Lambda)=\{e^{2\pi i b_k \cdot}g(\cdot - a_k)\}_{k=1}^{N}.$ Recall that the Gramian $G_g$ of $\mathcal{G}(g, \Lambda)=\{e^{2\pi i b_k \cdot}g(\cdot - a_k)\}_{k=1}^{N}$ is the matrix given by 
\begin{align}\label{coline-gram}
G_{g}&=(\ip{e^{2\pi i b_k\cdot}g(\cdot -a_k)}{e^{2\pi i b_\ell \cdot}g(\cdot - a_\ell)})_{k, \ell = 1}^N\\ \notag
&=(e^{-2\pi i a_k(b_\ell -b_k)}V_gg(a_\ell - a_k, b_\ell-b_k))_{k, \ell=1}^N.
\end{align}

It follows that $G_g$ is positive semidefinite matrix and that the HRT conjecture holds if and only if $G_g$ is strictly positive definite. In this section, we motivate our approach to analyze this Gramian by offering  a new proof of the  HRT conjecture for collinear points. While this result is well-known \cite{HRT96}, the new proof we provide illustrates the role of positive definiteness vis-a-vis the HRT conjecture. We will need the following version of Bochner's theorem, and refer to \cite[Theorem 4.18]{FollAHA95} for more on the classical  Bochner's theorem. We recall that a continuous complex-valued function $f:\R^d \to \C$ is positive definite if 
\begin{equation}\label{def-pos-def-fun}
\sum_{j=1}^N\sum_{k=1}^Nc_j\overline{c}_{k}f(x_j-x_k)\geq 0
\end{equation} for any pairwise distinct points $x_1,  x_2, \hdots, x_N\in \R^d$, and $(c_k)_{k=1}^N \in \C^N$. The function $f$ is said to be strictly positive if equality holds in~\eqref{def-pos-def-fun} only when $c_k=0$ for all $k=1, 2, \hdots, N$. 

\begin{prop}\cite[Proposition 2.1]{Derr10}\label{posdef}
A continuous  complex-valued function $f$ is positive definite if and only if $f=\hat{\mu}$ where $\mu$ is a non-negative finite Borel measure on $\R$. Furthermore, $f$ is strictly positive definite if and only if there does not exist a non-zero trigonometric polynomial $m$ vanishing on the support of $\mu$, i.e., such that  $\int_\R m d\mu=0$.
\end{prop}

Let $\Lambda=\{(a_k, b_k)\}_{k=1}^N \subset \R^{2}$ be a set of collinear points. Then, by rotating and translating we can assume that $\Lambda=\{(a_k, 0)\}_{k=1}^N \subset \R^{2}$ with $a_1=0$, \cite{HRT96}. In this case, the Gramian of $\mathcal{G}(g, \Lambda)$ takes the form 

\begin{equation*}
G_{g}=(\ip{g(\cdot -a_k)}{g(\cdot - a_\ell)})_{k, \ell = 1}^N= (\ip{\hat{g}}{e^{-2\pi i (a_k-a_\ell) \cdot}\hat{g}})_{k, \ell =1}^N.
\end{equation*}

We can now give a new proof of the HRT conjecture when the points are collinear. 

\begin{thm}\label{hrt-colinear}
Let $0\neq g \in L^2(\R)$, and $\Lambda=\{(a_k, 0)\}_{k=1}^N \subset \R^{2}$ with $a_0=0$. Then $$\mathcal{G}(g, \Lambda)=\{g(\cdot - a_k)\}_{k=1}^{N}$$ is linearly independent. 
\end{thm}

\begin{proof} For $g \in L^2(\R)$ with $\|g\|_2=1$ let $h(\xi)=|\hat{g}(\xi)|^2$. We note that $h$ is  non-negative, non-identically $0$, and $h \in L^1(\R)$. Let $\mu$ be the finite nonnegative Borel measure whose density with respect to the Lebesgue measure is the function $h$. The function $\Phi$ defined by $$\Phi(x)=\hat{h}(x)=\hat{\mu}(x)=\int_{\R}h(\xi) e^{-2\pi i x \cdot \xi}d\xi$$ is continuous. Consequently, by Proposition~\ref{posdef}, $\Phi$  is positive definite. It remains to show that $\Phi$ is strictly positive definite. To do this, suppose that $m$ is a non-zero trigonometric polynomial given by $m(x)=\sum_{k=1}^Kc_ke^{2\pi i \xi_k x}$ where $c_k$ are complex numbers (not all zeros), and $\xi_k$ are pairwise distinct real numbers.  It follows that 
\begin{align*}
\sum_{j=1}^K \sum_{k=1}^K \overline{c_j}\ c_k \ \Phi(\xi_j- \xi_k) &=\sum_{j=1}^K \sum_{k=1}^K \overline{c_j}\ c_k \ \int_\R e^{-2\pi i(\xi_j-\xi_k)x}\ d\mu(x)\\
& =\int_\R \ \sum_{j=1}^K\overline{c_j}\ e^{-2\pi i \xi_j x}\, \sum_{k=1}^Kc_k\ e^{2\pi i \xi_k x}\ d\mu(x)\\
&=\int_\R |m(x)|^2\ d\mu(x)\\
&=\int_\R |m(x)|^2\ |\hat{g}(x)|^2\ dx.
\end{align*} This last integral vanishes only when $m$ vanishes on the support of $\mu$, which is the support of $g$. We can now conclude that $\Phi=\hat{h}=\hat{\mu}$ is strictly positive definite. However, by~\eqref{coline-gram} we see that the Gramian of $ \{g(\cdot - a_k)\}_{k=1}^{N}$ is exactly the matrix 
$$G_{g}=(\ip{g(\cdot -a_k)}{g(\cdot - a_\ell)})_{k, \ell = 1}^N=(\Phi(a_k - a_\ell))_{k, \ell =1}^N.$$ 
Therefore,  $G_g$ is strictly positive definite. 
\end{proof}

\subsection{Motivation: The case of three points revisited}\label{subsec2.3}
We now motivate our approach using the analysis of the Gramian of $\mathcal{G}(g, \Lambda)=\{g(\cdot - a_k) e^{2\pi i b_k\cdot}\}_{k=1}^{3}$ for $3$ points $\{(a_k, b_k)\}_{k=1}^3 \subset \Z^2$. Furthermore, we suppose that the function $g$ is real-valued. 
We note that following \cite{HRT96}, without loss of generality any set of three distinct  points can be transformed (through area preserving transformations)  into $\{(0,0), (0, 1), (a, b)\}$ where $(a, b) \in \Z^2\setminus \{(0,0), (0,1)\}$. We also know that the HRT conjecture is always true for any set of two distinct points. Thus, $\{g, M_1g\}$ is linearly independent and our task is to show that for any other point $(a, b) \in \Z^2\setminus \{(0,0), (0,1)\}$, $\{g, M_1g, M_bT_ag\}$ remains  linearly independent. 

Observe that the Gramian $G_g$ of $\{g, M_1g, M_bT_ag\}$ can be written in the following block structure:

\begin{equation}\label{gram3}
G_g=\begin{bmatrix}A & u(a,b)\\ u(a,b)^* & 1\end{bmatrix}
\end{equation}

 where $$A=\begin{bmatrix} 1& \alpha \\ \overline{\alpha} &1\end{bmatrix}\qquad {\textrm and} \qquad u(a,b)=\begin{bmatrix}V_gg(a, b)\\ V_gg(a, b-1)\end{bmatrix}$$ with $\alpha=V_{g}g(0,1)$, and $u(a,b)^*$ denoting the conjugate adjoint of $u(a,b)$. Note that  $|\alpha| =|\ip{g}{M_{1}g}|<\|g\|_2\|M_1g\|_2=\|g\|_2^2=1$  since $\{g(\cdot) , e^{2\pi i \cdot}g(\cdot)\}$ is linearly independent. We know that $G_g$ is positive semidefinite and we wish to show that it is strictly positive definite. Appealing to Corolloary~\ref{version-psd}, we see that $$0\leq F(a,b)=\ip{A^{-1}u(a,b)}{u(a,b)}\leq 1$$  and that  $0\leq F(a,b)<1$ if and only if $\{g, M_1g, M_bT_ag\}$  is linearly independent. Thus, the function $F: \R^2 \to \R$ has range in $[0,1]$ and $1$ is its maximum value.

We can now prove the following result which serves both as a motivation to our approach and gives a new proof for the HRT conjecture for any $3$ points on the integer lattice.

\begin{prop}\label{L2-3}
Let $0\neq g \in L^2(\R)$ be a real-valued function with $\|g\|_2=1$, and $\Lambda=\{(0,0), (0, 1), (a,b)\}$, with $(a, b)\in \Z^2 \setminus \{(0,0), (0,1)\}$. Then the function $F$ defined above  achieves its global maximum value $1$ only for $(a, b) \in \{(0,0), (0,1)\}$. Consequently,
Conjecture~\ref{hrt} holds for  $\Lambda$ and $g$. 
\end{prop} 

The proof of  Proposition~\ref{L2-3}  is based Demeter's result on $(2,2)$ configurations, as well as on a   symmetry property of $F$. It also illustrates the restriction principle that will be introduced in Section~\ref{subsec3.3}.  First we prove the following symmetry of  $F$. 

\begin{lemma}\label{Sym:F} Suppose that $g\in L^2(\R)$ is real-valued. With the setting above we have   
$$F(a, b)=F(a, 1-b)$$ for all $(a, b)\in \Z \times \R$, $b\neq 1/2$, and $$F(-a, 1/2)=F(a,1/2)$$ for all $a \in \Z$. 
\end{lemma}

\begin{proof}
Let $$V=\begin{bmatrix}0& \ 1\\ 1&0\end{bmatrix}.$$  Then $F(a, 1-b)=\ip{A^{-1}u(a, 1-b)}{u(a, 1-b)},$ and 
$$
u(a, 1-b)=\begin{bmatrix}V_{g}g(a,1-b)\\V_gg(a, b)\end{bmatrix}=\overline{V \begin{bmatrix}V_{g}g(a,b)\\V_gg(a, b-1)\end{bmatrix}}=V\overline{u(a,b)}.$$ 
Moreover, straightforward computations show that $V^{T}A^{-1}V=VA^{-1}V=\bar{A}^{-1}.$ 

Consequently,
\begin{align*}
F(a, 1-b)&=\ip{A^{-1}u(a, 1-b)}{u(a, 1-b)}=\ip{A^{-1}V\overline{u(a,b)}}{V\overline{u(a,b)}}\\
&=\ip{V^TA^{-1}V\overline{u(a,b)}}{\overline{u(a,b)}}=\ip{\bar{A}^{-1}\overline{u(a,b)}}{\overline{u(a,b)}}\\
&=\overline{\ip{\bar{A}^{-1}\overline{u(a,b)}}{\overline{u(a,b)}}}=\ip{A^{-1}u(a,b)}{u(a,b)}\\
&=F(a, b)
\end{align*} where we have used the fact that $A^{-1}$ is a positive definite matrix.

When $b=1/2$ and $a\in \Z$, we see that $F(-a, 1/2)=\ip{A^{-1}u(-a, 1/2)}{u(-a, 1/2)}.$ but $$u(-a, 1/2)=\begin{bmatrix}V_gg(-a, 1/2)\\ V_gg(-a, -1/2)\end{bmatrix}=\begin{bmatrix}e^{\pi i a}V_gg(a, 1/2)\\ e^{-\pi i a}V_{g}g(a, -1/2)\end{bmatrix}=Bu(a,b)$$ where $B=\begin{bmatrix}e^{\pi i a} &0\\0 &e^{-\pi i a}\end{bmatrix}$. It follows that
\begin{align*}
F(-a,1/2)&=\ip{A^{-1}u(-a, 1/2)}{u(-a, 1/2)}=\ip{A^{-1}Bu(a, 1/2)}{Bu(a, 1/2)}\\
&=\ip{B^*A^{-1}Bu(a,1/2)}{u(a,1/2)}=\ip{A^{-1}u(a,1/2)}{u(a,1/2)}\\
&=F(a,1/2)
\end{align*} where we used the fact that  $B^*A^{-1}B=A^{-1}$. To see why this is the case we observe that $a\in \Z$ and a series of computations shows that
$$B^*A^{-1}B=\tfrac{1}{1-|\alpha|^2}\begin{bmatrix}1 & -\alpha e^{-2\pi i a}\\-\overline{\alpha}e^{2\pi i a}& 1\end{bmatrix}= \tfrac{1}{1-|\alpha|^2}\begin{bmatrix}1 & -\alpha \\-\overline{\alpha}& 1\end{bmatrix}=A^{-1}.$$

\end{proof}

The following argument gives an alternate  proof of Proposition~\ref{L2-3}.

\begin{proof}
Because $\{g, M_1g, g\}$ and $\{g, M_1g, M_1g\}$ are both linearly dependent (repeated vectors) it follows that  $F(0,0)=F(0,1)=1$. Assume that there exists $(a_0, b_0) \in \Z \times \R \setminus  \{(0,0), (0,1)\}$ such that $F(a_0, b_0)=1$ with $b_0\neq 1/2$.  In particular, by Corollary~\ref{version-psd} the system $\mathcal{G}(g, \{(0,0), (0,1), (a_0, b_0)\})$ is linearly dependent, i.e., $M_{b_{0}}T_{a_{0}}g$ belongs to the linear span of $\mathcal{G}(g, \{(0,0), (0,1)\})$. Using Lemma~\ref{Sym:F} we get that $F(a_0, 1-b_0)=1$ and therefore $M_{1-b_{0}}T_{a_{0}}g$ belongs to the linear span of $\mathcal{G}(g, \{(0,0), (0,1)\})$. Therefore,  $\mathcal{G}(g, \{(0,0), (0,1), (a_0, b_0), (a_0, 1-b_0)\})$ is also linearly dependent. But $\{(0,0), (0,1), (a_0, b_0), (a_0, 1-b_0)\}$ is a $(2,2)$ configuration and $g \in L^2(\R)$, which contradicts \cite[Theorem 1.4]{DemZah12}. 

The last case to consider is to assume that for some $a_0\neq 0$, $b_0=1/2$ and  $F(a_0, 1/2)=1$. But then using Lemma~\ref{Sym:F} again we see that $F(-a_0, 1/2)=F(a_0,1/2)$ and therefore $\mathcal{G}(g, \{(0,0), (0,1), (a_0, 1/2), (-a_0, 1/2)\})$ is linearly dependent. But $\{(0,0), (0,1), (a_0, 1/2), (-a_0, 1/2)\}$ is also a $(2,2)$ configuration and $g \in L^2(\R)$, which contradicts \cite[Theorem 1.4]{DemZah12}. 
\end{proof}

\section{Extension and restriction  principles to the HRT conjecture}
In this section we describe in its full generality the aforementioned extension principle for the HRT conjecture. It could also be viewed as an inductive approach to attack the conjecture. More specifically, suppose that the Conjecture holds for a given function $g \in L^{2}(\R)$ and a given set $\Lambda=\{(a_k, b_k)\}_{k=1}^{N} \subset \R^2$. We seek all the points $(a, b) \in \R^2 \setminus \Lambda$ such that the conjecture remains true for the same function $g$ and the new set $\Lambda'=\Lambda \cup \{(a,b)\}$.  We investigate this question by using Theorem~\ref{psd-matrix} and Corollary~\ref{version-psd} to relate the Gramians of $\mathcal{G}(g, \Lambda')$ and $\mathcal{G}(g, \Lambda)$. In Section~\ref{subsec3.1} we introduce the main technical tool to extend the HRT in the sense given above. Subsequently, in  Section~\ref{subsec3.2} we apply this approach to  $(1, n)$ configurations. Finally, in Section~\ref{subsec3.3} we introduce a related restriction principle that allows us to establish the HRT conjecture for a family of $4$ points and real-valued functions from knowing that the conjecture can be proved for a related family of symmetric $(2,3)$  configurations. To establish the latter result  we refine  Demeter's ``conjugate trick" arguments to handle this family of $(2,3)$  configurations.

\subsection{The HRT extension principle}\label{subsec3.1} 
Let $g \in L^2(\R)$ with $\|g\|_2=1$. Assume that 
Conjecture~\ref{hrt} holds for some  $\Lambda=\{(a_k, b_k)\}_{k=1}^{N} \subset \R^2$ with $(a_1, b_1)=(0,0)$. Let $\Lambda'=\{(a_k, b_k)\}_{k=1}^{N} \cup \{(a, b)\}$ for   $(a, b) \in \R^2$. 

The Gramian $G_{g, N+1}(a,b) $ of $\mathcal{G}(g, \Lambda')=\{e^{2\pi i b_k \cdot}g(\cdot - a_k)\}_{k=1}^{N} \cup\{e^{2\pi i b \cdot}g(\cdot -a)\}$ has  the following block structure:

\begin{equation}\label{gram-induction}
G_{N+1}:=G_{g, N+1}(a,b)=\begin{bmatrix} G_{N} & u_N(a,b)\\  u(a,b)^{*} & 1\end{bmatrix}
\end{equation}
 where $G_N:=G_{g, N}$ is the Gramian of $\{e^{2\pi i b_k \cdot}g(\cdot - a_k)\}_{k=1}^{N}$, $u_N(a,b)$ is a vector in $\C^{N}$ given by 
\begin{equation}\label{vector-induction}
u_N(a,b)=\begin{bmatrix}e^{2\pi i a_1b_1}V_{g}g(a, b)\\ e^{-2\pi i  a_2 (b-b_2)}V_gg(a-a_2, b-b_2)\\ e^{-2\pi i a_3(b- b_3)}V_gg(a-a_3, b-b_3)\\ \vdots \\ e^{-2\pi i  a_{N}(b-b_N)}V_gg(a-a_{N}, b-b_N)\end{bmatrix}=\begin{bmatrix} V_{g}g(a, b)\\ e^{2\pi i  a_2b_2}V_g(T_{a_{2}}M_{b_{2}}g)(a, b)\\ e^{2\pi i a_3b_3}V_g(T_{a_{3}}M_{b_{3}}g)(a, b)\\ \vdots \\ e^{2\pi i  a_{N}b_N}V_g(T_{a_{N}}M_{b_{N}}g)(a, b)\end{bmatrix}
\end{equation} and 
$ u_N(a,b)^{*}$ is the adjoint of  $u_N(a,b)$.  

Because $G_{N}$ is positive definite,  the function  $F_{N+1}:\R^2 \to [0, \infty]$ given by 
\begin{equation}\label{maintool}
F(a,b):=F_{N+1}(a, b)=\ip{G_{N}^{-1}u_N(a,b)}{u_N(a,b)}
\end{equation}
is well-defined.
For simplicity and when the context is clear, we write $u(a,b)$ for $u_N(a,b)$, and $F(a, b)$ for $F_N(a,b)$.  Note that when $N=2$ the function $F$ is simply the one introduced in the proof of Theorem~\ref{L2-3}.
 The following result summarizes the main properties of $F$.

\begin{thm}\label{keyfunest}
With the above notations assume that $G_{ N}$ is a positive definite $N\times N$ matrix.   Then,  the following statements hold.
\begin{enumerate}
\item[(i)] $0\leq F(a,b)\leq 1$ for all $(a, b) \in \R^2$, and moreover, $F(a_k, b_k)=1$ for each $k=1, \hdots, N$. 
\item[(ii)] $F$ is uniformly continuous and $\lim_{|(a,b)|\to \infty}F(a,b)=0$.
\item[(iii)] $\iint_{\R^{2}}F(a,b)da db=N$.
\item[(iv)] The Fourier transform  $\widehat{F}:\R^2 \to \C$  of $F$ given by $$\widehat{F}(\xi, \eta)=\iint_{\R^{2}}F(a,b)e^{-2\pi i (a\xi+ b\eta)}dadb,$$ is strictly positive definite, and integrable. 
\item[(v)] $\det{G_{g}(a,b)}=(1-F(a,b))\det{G_{N}}$.
\end{enumerate}
\end{thm}

\begin{proof}
\noindent {\bf (i)} The Gramian of $G_{g}(a,b)$ is positive semidefinite so by Corollary~\ref{version-psd} and the assumption that $G_{N}$ is positive definite, we conclude that $$0\leq F(a,b)\leq 1.$$ Moreover, for $(a,b)=(a_k, b_k)$ we know that the Gramian is positive semidefinite as the system is linearly dependent (one element is repeated twice). Thus, we get the moreover part of the result. 

\noindent {\bf (ii)} This follows easily as each coordinate of $u(a,b)$ is a uniformly continuous function that tends to $0$ at infinity.

\noindent {\bf (iii)} Suppose that $G_{g,N}^{-1}=(B_{i,j})_{i, j=1}^N$.  We can now write 
\begin{align*}
F(a,b)&=\ip{G_{N}^{-1}u(a,b)}{u(a,b)}\\
&= \sum_{k=1}^{N}\sum_{\ell=1}^{N}B_{k, \ell}(u(a,b))_{\ell}\, \overline{(u(a,b))_{k}}\\
&= \sum_{k=1}^{N}\sum_{\ell=1}^{N}B_{k, \ell} \, e^{2\pi i (a_\ell b_\ell-a_k b_k)}\, V_{g}(T_{a_{\ell}}M_{b_{\ell}}g)(a,b) \, \overline{ V_{g}(T_{a_{k}}M_{b_{k}}g)(a,b)  }.
\end{align*}
Integrating this last formula over $\R^2$ and using the orthogonality and covariance properties of the STFT, i.e.,~\eqref{ortho-Vgg} and~\eqref{stft-TF} we have 
$$
\iint_{\R^{2}}F(a,b) da\, db = \sum_{k, \ell=1}^N B_{k, \ell} \, e^{2\pi i (a_\ell b_\ell-a_k b_k)}\,   \iint_{\R^{2}} V_{g}(T_{a_{\ell}}M_{b_{\ell}}g)(a,b)\,  \overline{ V_{g}(T_{a_{k}}M_{b_{k}}g)(a,b)  } \, da\, db
$$
Evaluating the integral leads to
\begin{align*}
\iint_{\R^{2}} V_{g}(T_{a_{\ell}}M_{b_{\ell}}g)(a,b)\,  \overline{ V_{g}(T_{a_{k}}M_{b_{k}}g)(a,b)  } \, da\, db &=  \ip{V_{g}(T_{a_{\ell}}M_{b_{\ell}}g)}{V_{g}(T_{a_{k}}M_{b_{k}}g)}\\ \notag
&=\ip{T_{a_{\ell}}M_{b_{\ell}}g}{T_{a_{k}}M_{b_{k}}g}\, \ip{g}{g}\\ \notag
&= \ip{T_{a_{\ell}}M_{b_{\ell}}g}{T_{a_{k}}M_{b_{k}}g}
\end{align*}
Consequently,

\begin{align*}
\iint_{\R^{2}}F(a,b) da\, db &=   \sum_{k, \ell=1}^N B_{k, \ell} \, e^{2\pi i (a_\ell b_\ell-a_k b_k)}\, \ip{T_{a_{\ell}}M_{b_{\ell}}g}{T_{a_{k}}M_{b_{k}}g}\\ \notag
&=   \sum_{k, \ell=1}^N B_{k, \ell}\,  \ip{M_{b_{\ell}}T_{a_{\ell}}g}{M_{b_{k}}T_{a_{k}}g}\\ \notag
&=  \sum_{k, \ell=1}^N B_{k, \ell}\, (G_{N})_{\ell,k}\\ \notag
&=\tfrac{1}{{\textrm det} G_{N}}  \sum_{ \ell=1}^N \sum_{k=1}^N (-1)^{k+\ell} \, (G_{N})_{\ell,k}\, {\textrm det} G_{N}(\{\ell\}',\{k\}') \\ \notag
& = \tfrac{1}{\det {G_{N}}}  \sum_{ \ell=1}^N \det{G_{N}}\\ \notag
&= N \notag
\end{align*}
where we use the fact that $B_{k,\ell}=\tfrac{(-1)^{k+\ell }}{{\textrm det} G_{N}} \textrm{det} G_{N}(\{k\}',\{\ell\}')$. 

\noindent {\bf (iv)} This part follows from the fact that $F$ is nonnegative, not identically $0$, continuous, and integrable.  

Using~\eqref{FT-prod-stft},~\eqref{stft-TF}, and the notations set in part {\bf (iii)} we can compute $\hat{F}$ explicitly
\begin{align*}
\hat{F}(\xi, \eta) &=  \sum_{k=1}^{N}\sum_{\ell=1}^{N}B_{k, \ell} e^{2\pi i (a_\ell b_\ell-a_k b_k)} 
\mathcal{F}_2(V_{g}(T_{a_{\ell}}M_{b_{\ell}}g)\,  \overline{ V_{g}(T_{a_{k}}M_{b_{k}}g) })(\xi, \eta)\\
&=\sum_{k=1}^{N}\sum_{\ell=1}^{N}B_{k, \ell} \, e^{2\pi i (a_\ell b_\ell-a_k b_k)} 
  V_{T_{a_{k}}M_{b_{k}}g}(T_{a_{\ell}}M_{b_{\ell}}g)(-\eta, \xi) \, \overline{ V_{g}g (-\eta, \xi)}\\
&=\sum_{k=1}^{N}\sum_{\ell=1}^{N}B_{k, \ell} \, e^{2\pi i (a_\ell b_\ell-a_k b_k- a_\ell b_k)}  \, e^{-2\pi i(a_\ell \xi +b_k\eta)}\,  V_{g}g(-\eta -a_\ell +a_k, \xi -b_\ell +b_k)\,  \overline{ V_{g}g (-\eta, \xi)}.
\end{align*}

It is clear that $$\iint_{\R^2}|\hat{F}(\xi, \eta)|d\xi d\eta \leq \sum_{k=1}^{N}\sum_{\ell=1}^{N}|B_{k, \ell}| < \infty.$$

\noindent {\bf (v)} Follows from Corollary~\ref{version-psd}.

%
%

\end{proof}

The following result is a consequence of Theorem~\ref{keyfunest}.

\begin{cor} \label{Fandhrt} Let $g \in L^2(\R)$ with $\|g\|_2=1$ and $\Lambda=\{(a_k, b_k)\}_{k=1}^{N} \subset \R^2$. Assume that $\mathcal{G}(g, \Lambda)$ is linearly independent. Let  $\Lambda'=\{(a_k, b_k)\}_{k=1}^{N} \cup \{(a, b)\}$. Then $\mathcal{G}(g, \Lambda')$ is linearly independent if and only if $F(a,b)<1$.  Furthermore,  there exists $R:=R(\Lambda, g)>0$ such that for all $(a, b) \in \R^2$ with $|(a, b)|> R$, then $\mathcal{G}(g, \Lambda')$ is linearly independent where $\Lambda'=\Lambda \cup \{(a, b)\}$
\end{cor}

\begin{proof}
The first part follows from   part (1) of  Corollary~\ref{version-psd} and part (i) of Theorem~\ref{keyfunest}. The existence of $R$ is guaranteed by part (ii) of Theorem~\ref{keyfunest}. 
\end{proof}

\begin{rem}\label{hrt-is-local}

\noindent {\bf (a)} By the last part of Corollary~\ref{Fandhrt}, the extension function $F$ makes the HRT conjecture a ``local problem''. In other words, once the conjecture is known to be true for a function $g$ and a set $\Lambda=\{(a_k, b_k)\}_{k=1}^N$, it is also automatically true  for $\Lambda'=\Lambda\cup \{(a, b)\}$ whenever the new point lies outside a ball of radius $R$. So to establish the HRT everywhere for $\Lambda'$ we must focus on the ``local'' properties of $F$, that is the restriction of $F$ to the aforementioned ball.

\noindent {\bf (b)}  Corollary~\ref{Fandhrt} makes it possible to explore the HRT from a numerical point of view. Indeed, Theorem~\ref{keyfunest} and 
Corollary~\ref{Fandhrt}  assert that the HRT for $\Lambda'=\{(a_k, b_k)\}_{k=1}^{N} \cup \{(a, b)\}$ holds if and only if $(a, b)$ is not a global maximizer of $F$. So in theory, one only needs to prove that the set of global maximizers of $F$ is  $\Lambda$. For  a smooth function $g$, differential calculus can be used to check this.  For example, for the  Gaussian  $g(x)=2^{1/4} e^{-\pi x^2}$, using \cite[Lemma 1.5.2]{Groc2001} we get that $$V_gg(a,b)=e^{-\pi i a b}e^{-\pi a^2/2}e^{-\pi b^2/2}.$$ 
In this case and using $\Lambda=\{(0,0), (0,1)\}\cup \{(a, b)\}$,  $F$ is simply
$$F(a,b)= \tfrac{e^{-\pi (a^2+b^2)}}{1-e^{-\pi}}[ 1 + e^{\pi (2b-1)} - 2 e^{\pi(b-1)}\cos a ].$$ One can then used multivariable calculus to show that the global maximizers of $F$  are exactly the two points $(0,0)$ and $(0,1)$, see Figure~\ref{Fig:fig3}.

\begin{figure}[ht]
\begin{center}
\includegraphics[scale=0.5]{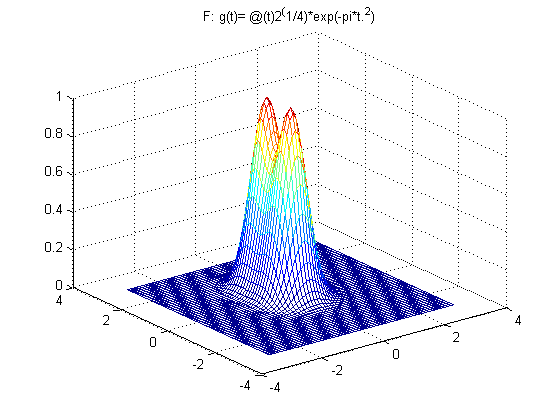}
\end{center}
\caption{Graph of $F$ for $\Lambda=\{(0,0), (0,1)\}$ and $g(x)=2^{1/4}e^{-\pi x^2}$.}
\label{Fig:fig3}
\end{figure}

More generally,  one can numerically analyze the function $F$ to determine its global maximizers. For example, suppose $\Lambda=\{(0,0), (1,0), (0,1)\}$. For the Gaussian $g(x)=2^{1/4}e^{-\pi x^2}$, Figure~\ref{Fig:fig4} displays the graph of the function $F$ on the square $[-4, 4]\times [-4,4]$. This graph illustrates the validity of the HRT in this case, by showing that the global maximum value of $F$ is only achieved on the set $\Lambda$.  Similarly, when $g(x)=e^{-|x|}$,  Figure~\ref{Fig:fig5} displays the graph of the function $F$ on the square $[-4, 4]\times [-4,4]$. This graph illustrates the validity of the HRT in this case, by showing that the global maximum value of $F$ is only achieved on the set $\Lambda$. Recall that the HRT is known to be true for this function and any set of $4$ points \cite{BeBo13}. Finally, Figure~\ref{Fig:fig6} displays the graph of the function $F$ on the square $[-4, 4]\times [-4,4]$ when $g(x)=\tfrac{2^{-1/2}}{1+|x|}$. To the best of our knowledge, the HRT has not been proved for this function and any set of $4$ points. Therefore, Figure~\ref{Fig:fig6} offers some numerical evidence to the validity of the conjecture in this case. We also refer to Corollary~\ref{4points-hrt} for some new results in this setting.

\begin{figure}[ht]
\begin{center}
\includegraphics[scale=0.5]{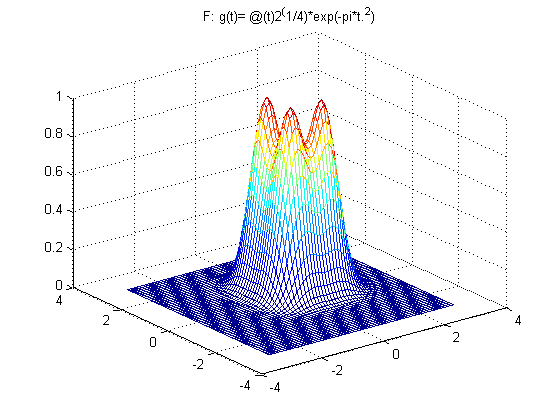}
\end{center}
\caption{Graph of $F$ for $\Lambda=\{(0,0), (0,1), (1,0)\}$ and $g(x)=2^{1/4}e^{-\pi x^2}$.}
\label{Fig:fig4}
\end{figure}

\begin{figure}[ht]
\begin{center}
\includegraphics[scale=0.5]{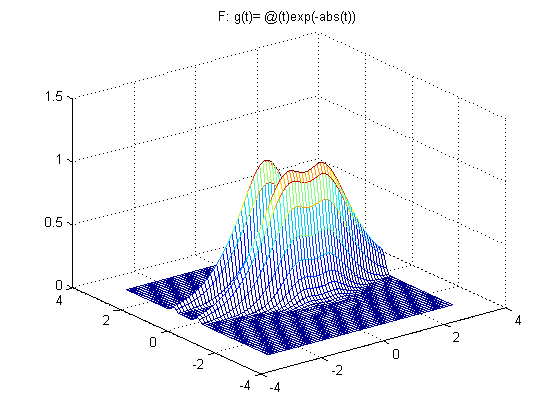}
\end{center}
\caption{Graph of $F$ for $\Lambda=\{(0,0), (0,1), (1,0)\}$ and $g(x)=e^{-|x|}$.}
\label{Fig:fig5}
\end{figure}

\begin{figure}[ht]
\begin{center}
\includegraphics[scale=0.5]{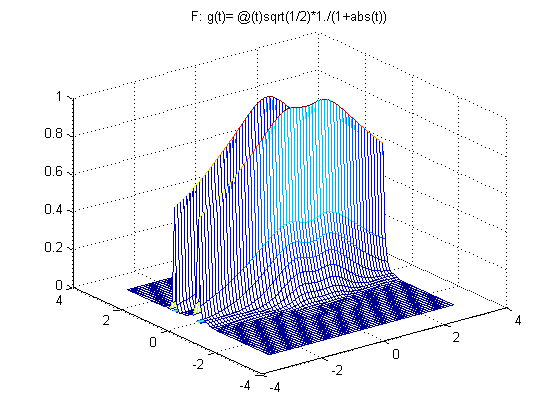}
\end{center}
\caption{Graph of $F$ for $\Lambda=\{(0,0), (0,1), (1,0)\}$ and $g(x)=\tfrac{2^{-1/2}}{1+|x|}$.}
\label{Fig:fig6}
\end{figure}
\end{rem}

\subsection{The HRT conjecture  for $(1, n)$ configurations}\label{subsec3.2}
In this section, we consider the HRT conjecture for $(1,n)$ configurations and prove that  the conjecture can only fail for at most one such configuration. The proof is elementary and based on some dimension arguments. We then focus on the case $n=3$ and show that when the generator is a real-valued function then  Conjecture~\ref{hrt} holds for all $(1,3)$.   Note that the strongest known results for these configurations assume either that the $3$ collinear points are also equi-spaced, or that the generator is in $S(\R)$. The motivation of the results presented in this section is \cite[Theorem 1.3]{wliu16} which states that the HRT conjecture holds for almost all (in the sense of Lebesgue measure) $(1,3)$ configurations. A consequence of our result is that the HRT conjecture can only fail for at most one $(1,3)$ configuration up to equivalence.  For more on the HRT for $(1,3)$ we refer to Demeter's results \cite{Dem10} and a recent improvement due to Liu \cite{wliu16}.

Recall that by using the metaplectic transformations one can show that any $(1,n)$ configuration has the form $\{(0,1)\}\cup \{(a_k, 0)\}_{k=1}^n$ where $a_1=0$ and the rest of the  $a_k$s are distinct and nonzero \cite{HRT96}. Note that the set of metaplectic transformations in $\R^2$ can be identified with the set of $2\times 2$ symplectic matrices, which, in turn is $SL(2, \R)$ \cite{FollHAPS89, Groc2001}. We say that two $(1, n)$ configurations $\Lambda_1$ and $\Lambda_2$ are equivalent if and only if there exists a symplectic matrix $A\in SL(2, \R)$ such that $\Lambda_2=A\Lambda_1$. Let the set of distinct equivalence classes under this relation be denoted by $\Lambda_{(1,n)}$. Without any loss of generality we can assume that 
$$\Lambda_{(1,n)}=\{(0,1)\}\cup\{(a_k, 0)\}_{k=1}^n$$ with $a_1=0$ and $a_k\neq 0$ for all $k=2, 3, \hdots, n.$ To prove that the HRT conjecture holds for all $(1,n)$ configurations, it is enough to restrict to $(1,n)$ configurations in $\Lambda_{(1,n)}$.

\begin{thm}\label{1-nconfig}
Let $n\geq 3$ and $g\in L^2(\R)$ with $\|g\|_2=1$. Suppose that the HRT conjecture holds for $g$ and any $(1, n-1)$ configuration. Then there exists at most one (equivalence class of) $(1, n)$ configuration $\Lambda_0 \in \Lambda_{(1,n)}$ such that $\mathcal{G}(g, \Lambda_0)$ is linearly dependent. Furthermore, suppose that $\Lambda_0=\{(0,1)\}\cup \{(a_k, 0)\}_{k=1}^n \in \Lambda_{(1,n)}$ is a $(1,n)$ configuration such that $\mathcal{G}(g, \Lambda_0)$ is linearly dependent. Let $a\neq  a_k$  for $k=1, \hdots, n$. Fix any $k_0 \in \{1, \hdots, n\}$ and consider $$\Lambda =\{(0,1)\}\cup \{(a_1, 0), (a_2, 0), \hdots, (a_{k_{0}-1}, 0), (a, 0), (a_{k_{0}+1}, 0), \hdots, (a_n, 0)\}.$$ Then $\mathcal{G}(g, \Lambda)$ is linearly independent.
\end{thm}

\begin{proof}  Suppose by contradiction that there exist two distinct  $(1, n)$ configurations (or equivalent classes) $\Lambda_1$ and $\Lambda_2$ such that $\mathcal{G}(g, \Lambda_i)$ is linearly dependent for $i=1,2$.  Further, suppose that $$\Lambda_1=\{(0,1)\}\cup \{(a_k, 0)\}_{k=1}^n \quad {\textrm and }\quad \Lambda_2=\{(0,1)\}\cup \{(b_k, 0)\}_{k=1}^n$$ where $a_1=b_1=0$ and $a_{i_{0}} \neq b_{i_{0}}$ for some $i_{0} \in \{2, \hdots, n\}$. Then, one can write 
\begin{equation}\label{m1g}
M_1g=\sum_{k=1}^nc_kT_{a_{k}}g
\end{equation} where $c_k \neq 0$ for each $k=1, \hdots, n$. Indeed, if $c_{\ell}=0$ for some $\ell \in \{1, \hdots, n\}$ then $\Lambda'_1=\Lambda_1\setminus \{(a_{\ell}, 0)\}$ will be a $(1, n-1)$ configuration and~\eqref{m1g} will become $$M_1g=\sum_{k=1,\,  k \neq \ell}^nc_kT_{a_{k}}g.$$  That is $\mathcal{G}(g, \Lambda_1')$ will be linearly dependent contradicting one of the assumptions of the Theorem. Similarly,  
\begin{equation}\label{m1gbis}
M_1g=\sum_{k=1}^nd_kT_{b_{k}}g
\end{equation} where $d_k \neq 0$ for each $k=1, \hdots, n$. Taking the difference between~\eqref{m1g} and~\eqref{m1gbis} and rearranging leads to $$ (c_1-d_1) g + c_{i_{0}}T_{a_{i_{0}}}g - d_{i_{0}}T_{b_{i_{0}}}g+\sum_{k=2, \, k\neq i_0}^n c_kT_{a_{k}}g - \sum_{k=2, \, k\neq i_0}^n d_kT_{b_{k}}g =0$$ where $i_0$ was chosen above. But since $c_{i_{0}}d_{i_{0}}\neq 0$ and  $a_{i_{0}} \neq b_{i_{0}}$, this last equation is equivalent to the fact that $$\{g, T_{a_{k}}g, T_{b_{k}}g: k=2\, \hdots, n\}$$ is linearly dependent. But this contradicts the fact the HRT conjecture holds for any $0\neq g\in L^{2}(\R)$ and the collinear points $\{(0,0), (a_k, 0), (b_k, 0): k=2, \hdots, n\}$, \cite{HRT96}. Therefore, there can exist at most one (class of equivalence) $(1, n)$ configuration $\Lambda_0$ for which 
$\mathcal{G}(g, \Lambda_0)$ is linearly dependent.

For the last part,  suppose that $\Lambda_0=\{(0,1)\}\cup \{(a_k, 0)\}_{k=0}^n \in \Lambda_{(1,n)}$ is such that $\mathcal{G}(g, \Lambda_0)$ is linearly dependent.  Write   $\Lambda_0= \Lambda_0'\cup  \{(a_{k_{0}}, 0)\}$ where $$\Lambda_0'=\{(0,1), (0,0), (a_2, 0), \hdots, (a_{k_{0}-1}, 0), (a_{k_{0}+1}, 0), \hdots, (a_n, 0)\}.$$  By assumption, $\mathcal{G}(g, \Lambda_0')$ is linearly independent since $\Lambda_0'$ is a $(1, n-1)$ configuration. Then by  Corollary~\ref{Fandhrt}, $F(a_{k_{0}},0)=1$ where $F$ is the function obtained from the Gramian of $\mathcal{G}(g, \Lambda_0)$ according to Theorem~\ref{keyfunest}.

Now, let  $a\not \in \{0, a_{k}: k=2, \hdots, n\}$. If $F(a, 0)=1$ then $T_ag $ must belong to the linear span of $\{M_1g, T_{a_{k}}g: k=1, \hdots, n,\, k\neq k_{0}\}$ whose dimension is $n$. But, $T_{a_{k_{0}}}g $ also belongs to this linear span. Therefore, the $n+1$ functions $g, T_{a_{k}}g, T_ag,$ $k=2, \hdots, n$ belong to an $n$ dimensional space. However, these functions are linearly independent (because the points are collinear). This is a contradiction, from which we conclude that $F(a, 0)<1$, concluding the proof.

\end{proof} 

In the special case where $n=3$ we have the following result.
\begin{cor}\label{gen1-3config}

Let $g\in L^2(\R)$ with $\|g\|_2=1$. There exists at most one (equivalence class of) $(1, 3)$ configuration $\Lambda_0 $ such that $\mathcal{G}(g, \Lambda_0)$ is linearly dependent.
\end{cor}

\begin{proof}
When $n=3$ it is known that the HRT conjecture holds for $g$ and every $(1, 2)$ configuration \cite{HRT96}. Thus the assumption of Theorem~\ref{1-nconfig} is satisfied and the corollary follows.
\end{proof}

If we restrict to real-valued functions, then we can prove a stronger result by ruling out the existence of the single ``bad'' (equivalence class of ) $(1,3)$ configuration given by Corollary~\ref{gen1-3config}.

\begin{thm}\label{1-3config}
Let $g\in L^2(\R)$,  $\|g\|_2=1$ be a real-valued function. Let  $a\neq b\neq 0$ and set $\Lambda=\{(0,0), (0,1), (a,0), ( b,0)\}$  be a $(1, 3)$ configuration. Then, Conjecture~\ref{hrt} holds for  $\Lambda$ and $g$.
\end{thm}

\begin{proof}
Assume by way of contradiction that $\mathcal{G}(g, \Lambda)$ is linearly dependent. Then, there exists $c_k \in \C^{*}$, $k=1, 2, 3,$ such that $$c_1g + c_2M_1g + c_3T_ag = T_bg. $$ Because, $g$ is real-valued, we see that $$\overline{T_bg}=T_bg= \bar{c}_1g + \bar{c}_2M_{-1}g + \bar{c}_3T_ag.$$ Hence, $$(c_1-\bar{c}_1)g+ c_2M_1g - \bar{c}_2M_{-1}g +(c_3-\bar{c}_3)T_ag=0.$$ Note that $c_2\neq 0$. Hence, this last equation  is equivalent to the fact  that  $\mathcal{G}(g, \Lambda')$  where $\Lambda'=\{(0,0), (a, 0), (0,1), (0, -1)\}$ is linearly dependent. However, because the points $(0,1), (0,0)$ and $(0, -1)$ are equally spaced, $\Lambda'$ is a $(1, 3)$ configuration, for which Conjecture~\ref{hrt} is known to hold \cite{HRT96}. Therefore, we arrive at a  contradiction.
\end{proof}

\subsection{A restriction principle for the HRT conjecture}\label{subsec3.3}

The goal of this section is to establish Conjecture~\ref{hrt} for a large family of sets of cardinality $4$ (that are not $(1,3)$ nor $(2,2)$ configurations) when $g$ is a real-valued function. In addition, we establish similar results for Conjecture~\ref{hrt-s}. 
In fact, we prove that the general case for (almost) any $4$ points follows from a special family of $(3,2)$ configurations. This is our restriction principle: proving that Conjecture~\ref{hrt} or Conjecture~\ref{hrt-s} hold for this special special family of $(3,2)$ configurations implies its validity for a large family  $4$ points.  The proof of the next result is an extension of Demeter's ``conjugate trick'' arguments \cite[Theorem 1.5 (b)]{Dem10}.

\begin{thm}\label{3-2config}
Let $g\in L^2(\R)$ with $\|g\|_2=1$. Suppose $\Lambda$ is a $(3,2)$ configuration given by $\Lambda=\{(0,0), (0,1), (0,-1), (a, b), (a, -b)\}$ where $b\neq 0$. Then, Conjecture~\ref{hrt} holds for  $\Lambda$ and $g$ whenever any of the following holds
\begin{enumerate}
\item[(i)] $a, b \in \Q$.
\item[(ii)] $a\in \Q$ but $b\not\in \Q$.
\item[(iii)] $a, b \not\in \Q$ but $ab \in \Q$, and $g$ is a real-valued function. 
\end{enumerate}
\end{thm} 

\begin{proof}
We can  trivially assume that $a\neq 0$. Indeed, if $a=0$ then the points in $\Lambda$ will all lie on the $y$-axis, that is the points will be collinear, and HRT is known to be true in this case \cite{HRT96}. 

\vspace{.1in}

\noindent {\bf (i)} Suppose that $a=\tfrac{p}{q}, b=\tfrac{m}{n} \in \Q$.  In this case, we see that $\Lambda=A\Lambda'$, where $A=\begin{bmatrix}\tfrac{1}{q}&0\\0&\tfrac{1}{n}\end{bmatrix}$ and $\Lambda'=\{(0,0), (0, n), (0, -n), (p, m), (p, -m)\}.$ In particular, $\Lambda$ is a subset of a lattice and the result follows from \cite{Lin99}. 

\vspace{.1in}

\noindent {\bf (ii)} Next assume that $a \in \Q$ and $b\not\in \Q$.  By using a scaling matrix (a metaplectic transform)  we can assume that $\Lambda$ has the following form:
$$\Lambda=\{(0,0), (0,a), (0,-a), (1, b'), (1, -b')\}$$ with $ b'=ba \not\in \Q$ \cite{Heil06}.  To simplify the notations we will assume that $$\Lambda=\{(0,0), (0,a), (0,-a), (1, b), (1, -b)\}$$ with $a \in \Q$ and $ b \not\in \Q$. 

Assume by  way of contradiction that $\mathcal{G}(g, \Lambda)$ is linearly dependent. Then, there exist $c_k \in \C$, $k=1, 2, 3$, and $d_k\in \C$, for $k=1, 2$ such that 
\begin{equation}\label{5lidep}
c_1g + c_2M_ag + c_3M_{-a}g= d_1M_{-b}T_1g + d_2M_{b}T_1g.
\end{equation}
 Observe that $c_k, d_k\neq 0$ for each $k$, since the conjecture is true for all  $(2,2)$ configurations, and $(1,3)$ configurations where the points on the line are equiangular. We may also assume that $c_1\in \R$. 
 
 Consequently, we can write~\eqref{5lidep} as 
 
 \begin{equation}\label{recur5}
 |P(x)g(x)|=|Q(x)g(x-1)|\quad a.\ e.
 \end{equation}
where $P(x)=c_1+c_2e^{2\pi i ax} +c_3e^{-2\pi i ax} $ and $Q(x)=e^{2\pi i (-bx+\theta)}(r_1+r_2e^{2\pi i (2bx+\theta')}$ with $r_1, r_2\in (0, \infty)$ and $\theta, \theta'\in [0, 1)$. (Here we write $d_1=r_1e^{2\pi i \theta}$ and $d_2=r_2e^{2\pi i \theta'}$.) 
 
Furthermore, because $0\neq g \in L^2(\R)$ we have that
\begin{equation}\label{decayinfty}
\lim_{|n|\to \infty \;  n\in \Z}g(x-n)=0 \quad a.\ e.
\end{equation} 
and that  $\textrm{supp} (g)\cap [0,1]$ has a positive measure. Let $S\subset \textrm{supp}{g}\cap [0,1]$ be such that $S$ has positive measure, and such that $S+\Z$ contains no zeros of $P$ and $Q$ (this is possible since the set of such zeros is at most countable). From now on,  we assume that ~\eqref{recur5} and ~\eqref{decayinfty} hold for all $x\in S$. 

Next, by the Birkhoff's pointwise ergodic theorem with $1_{S}$, there exists $x_0\in S$ and $n' \in \N$ such that $x_1=\{-x_0-\tfrac{\theta'}{b}+ \tfrac{n'}{b}\} \in S$. Here and in what follows, we denote a fractional part of $x\in \R$ by $\{x\}$.  Let $m=-x_0-\tfrac{\theta'}{b}+ \tfrac{n'}{b}-x_1=y-x_1$.
 
By iterating~\eqref{recur5},  it follows that for all $N>m$
 \begin{equation}\label{iter-left-right}
 \left\{ \begin{array} {r@{\quad = \quad}l}
 |g(x_0+N)|&|g(x_0-1)|\tfrac{\prod_{n=0}^N|Q(x_0+n)|}{\prod_{n=0}^{N}|P(x_0+n)|}\\
 |g(x_1-N+m)|&|g(x_1-1)|\tfrac{\prod_{n=-N+m+1}^{-1}|P(x_1+n)|}{\prod_{n=-N+m+1}^{-1}|Q(x_1+n)|}
 \end{array}\right.
 \end{equation}
 Next, observe that $$\prod_{n=-N+m}^m|Q(x_1+n)|=\prod_{n=0}^N|Q(x_0+n)|.$$ 
 Consequently,
 \begin{align*}
 \prod_{n=-N+m}^{-1}|Q(x_1+n)|&=\prod_{n=-N+m}^{m}|Q(x_1+n)| \tfrac{1}{\prod_{n=0}^m|Q(x_1+n)|}& =K \prod_{n=-N+m}^{m}|Q(x_1+n)|\\
 &=K\prod_{n=0}^{N}|Q(x_0+n)|,
 \end{align*} where $K=\tfrac{1}{\prod_{n=0}^m|Q(x_1+n)|}$ is a positive finite constant that depends only on $m, x_0, n', b,$ and $ \theta'$.
 
 Now assume that $a=t/s \in \Q$, then $P$ is $s-$periodic. Let $T(x)=\prod_{n=0}^{s-1}|P(x+n)|$, and assume first that $T(x_1)\geq T(x_0)$. Then,  
 $$\prod_{n=-N+m+1}^{-1}|P(x_1+n)|=K' \prod_{n=0}^{N}|P(x_1 -n)|\geq \prod_{n=0}^{N}|P(x_0+n)|$$ for all $N>m$, where $K'=\tfrac{1}{|P(x_1)|\prod_{n=-N}^{-N+m}|P(x_1+n)|}$ is a constant independent of $N$.

  Consequently, for $N>m$, 
 \begin{align*}
  |g(x_1-N+m)|&=|g(x_1 -1)|\tfrac{\prod_{n=-N+m}^{-1}|P(x_1+n)|}{\prod_{n=-N+m}^{-1}|Q(x_1+n)|}\\
  &\geq C |g(x_1 -1)| \tfrac{\prod_{n=0}^N|P(x_0+n)|}{\prod_{n=0}^N|Q(x_0+n)|}\\
  &\geq C |g(x_1-1)||g(x_0-1)| |g(x_0+N)|^{-1}
  \end{align*} where $C$ is a constant that depends only on $x_0, m, r_1, r_2, c_1, c_2,$ and $c_3$. But this last inequality contradicts~\eqref{decayinfty}. 
  
  Now if instead, $T(x_0)\geq T(x_1)$. We will have 
  $$\prod_{n=-N+m}^{-1}|P(x_0+n)| \geq \prod_{n=0}^{N}|P(x_1+n)|$$ for all $N>m$, 
  $$\prod_{n=-N+m}^{-1}|Q(x_0+n)| \simeq \prod_{n=0}^{N}|Q(x_1+n)|,$$ where we used the notation $A\simeq B$ to denote $B/c\leq A\leq cB$ for some constant $c$ that depends only on $x_0, m, r_1,$ and $r_2.$
  
For $N>m$, $$ |g(x_0-N+m)|\geq C |g(x_0-1)| \tfrac{\prod_{n=0}^{N}|P(x_1+n)|}{\prod_{n=0}^{N}|Q(x_1+n)|},$$ and $$|g(x_1+N)|=|g(x_1-1)|  \tfrac{\prod_{n=0}^{N}|Q(x_1+n)|}{\prod_{n=0}^{N}|P(x_1+n)|}.$$ 
Consequently, for $N>m$, 
$$|g(x_0-N+m)|\geq C |g(x_0-1)| |g(x_1-1)||g(x_1+N)|^{-1},$$ where $C$ is a constant that depends only on $x_0, m, r_1, r_2, c_1, c_2,$ and $c_3$. But this last inequality contradicts~\eqref{decayinfty}. 

We conclude that~\eqref{5lidep} cannot hold unless, $c_k=0$ for $k=1, 2,3$ and $d_k=0$ for $k=1,2$.

\vspace{.1in}

\noindent {\bf (iii)} Similar to case (ii), and using a metaplectic transform we can assume that $\Lambda$ is of the form $\Lambda=\{(0,0), (0,a), (0,-a), (1, b), (1, b)\}$ with $a\not\in \Q, b\in \Q$. 

We now proceed as in part {\bf (ii)} and assume that~\eqref{5lidep} holds. Since, $g$ is assumed to be real-valued we see by taking the complex conjugate of~\eqref{5lidep} that 
 
 $$c_1g+\overline{c_2}M_{-a}g +\overline{c_3}M_{a}g-\overline{d_1}M_{-b}T_1g-\overline{d_2}M_{b}T_1g=0.$$ taking the difference between this last equation and~\eqref{5lidep} , we obtain 
 $$(c_2-\overline{c_3})M_ag  + (c_3-\overline{c_2})M_{-a}g +(\overline{d}_2-d_1)M_bT_1g +(\overline{d}_1-d_2)M_{-b}T_1g=0.$$ Now, the points $\{(0,a), (0,-a), (1, b), (1, -b)\}$ form a $(2,2)$ configuration and the HRT conjecture is true in this case. Therefore, 
 $c_3=\bar{c}_2, d_2=\bar{d}_1$. Consequently, we let $c_1=c\in \R, c_2=re^{2\pi i \theta},$ and $ d_1=r'e^{2\pi i \theta'},$ where $r, r'\in (0, \infty)$ and $\theta, \theta' \in [0,1)$. 

Therefore,~\eqref{recur5} holds with $P(x)=c+2r\cos 2\pi (ax +\theta)$ and $Q(x)=2r'\cos 2\pi (bx +\theta').$ In particular, $Q$ is a $s-$periodic function if we let $b=t/s\in \Q$. Reversing the role of the polynomials $P$ and $Q$ in the proof of (ii) establishes the result in this last case.

\end{proof}

\begin{rem} We note that the case $a\not\in \Q$ and $b\in \Q$ is equivalent (by a metaplectic transformation) to $a, b, ab\not\in \Q$. This is the only case we have not been able to address. However, if we assume that $g$ is smoother, then we can handle this case as well, see Theorem~\ref{smooth-3-2config} below. 

\end{rem}

We can now prove the following result for a family of  $4$ points in $\R^2$ and real-valued functions. This illustrates the restriction principle we announced in the introduction. Indeed, to establish the HRT conjecture for the family of sets of four points we use the fact the conjecture was proved for the above family of symmetric $(2,3)$ configurations. More specifically, the following result holds. Note that any set of four distinct points can be transformed into $\{(0,0), (0, 1), (s, 0), (a, b)\}.$

\begin{cor}\label{4points-hrt}
Let $g\in L^2(\R)$, $\|g\|_2=1$  be a real-valued function. Suppose  that $\Lambda= \{(0,0), (0, 1), (s, 0), (a, b)\} \subset \R^2$ be a subset of four distinct points.  Conjecture~\ref{hrt} holds for  $\Lambda$ and $g$, whenever any of the following holds
\begin{enumerate}
\item[(i)] $a, b\in \Q$.
\item[(ii)] $a\in \Q$ but $b\not\in \Q$.
\item[(iii)] $a, b \not\in \Q$ but $ab \in \Q$
\end{enumerate}
\end{cor}

\begin{proof}
 If $ab=0$ then, we are done by invoking Theorem~\ref{1-3config}. So we assume that $ab\neq 0$, and suppose by contradiction that there exist nonzero coefficients $c_1, c_2, c_3$ such that $$T_sg=c_1g+c_2M_1g+c_3M_{b}T_ag.$$ This implies that $$T_{s}g=\bar{c}_1g+\bar{c}_2M_{-1}g+\bar{c}_3M_{-b}T_{a}g.$$ Hence, $$ (c_1-\overline{c}_1)g + c_2M_1g -\overline{c}_2M_{-1}g + c_3M_{b}T_ag   - \bar{c}_3M_{-b}T_{a}g=0. $$ Consequently,  $\mathcal{G}(g, \Lambda)$ with $\Lambda=\{(0,0), (0,1), (0,-1), (a, b), (a, -b)\}$ is linearly dependent, which, contradicts Theorem~\ref{3-2config}. 
\end{proof}

\begin{example}\label{chris-example}
We recall the following conjecture. 
 \begin{conjecture}\cite[Conjecture 9.2]{Heil06}. Suppose $\Lambda=\{(0,0), (0,1), (1, 0), (\sqrt{2}, \sqrt{2})\}$, and if $0\neq g \in L^{2}(\R)$ with $\|g\|_2=1$ then $\mathcal{G}(g, \Lambda)$ is linearly independent. 
 \end{conjecture}

An application of Corollary~\ref{4points-hrt} settles this conjecture in the special case where $g$ is real-valued. Indeed, this follows from part {\bf (iii)} of Corollary~\ref{4points-hrt}  by taking $s=1$, $a=b=\sqrt{2}\not\in \Q$.
 \end{example}

If we assume that the  function $g$ is smoother, i.e., $g\in \S(\R)$ then we can extend \cite[Theorem 1.3]{Dem10} from $(2,2)$ configurations to certain symmetric $(3,2)$ configurations. It must be noted that the arguments given below were originally  introduced in \cite[Theorem 1.3]{Dem10}\footnote{The proof given in \cite[Theorem 1.3]{Dem10} contains a few inaccuracies that were fixed by C.~Demeter and posted on Math Arxiv as arXiv:1006.0732.}. For the sake of completeness we give the details of the proof below.

\begin{thm}\label{smooth-3-2config}
Let $g\in \S(\R)$, $\|g\|_2=1$. Suppose $\Lambda$ is a $(3,2)$ configuration given by $\Lambda=\{(0,0), (0,1), (0,-1), (a, b), (a, -b)\}$ where $b\neq 0$. Then, Conjecture~\ref{hrt-s} holds for  $\Lambda$ and $g$  whenever any of the following holds
\begin{enumerate}
\item[(i)] $a, b \in \Q$.
\item[(ii)] $a\in \Q$ but $b\not\in \Q$.
\item[(iii)] $a, b, ab \not\in \Q$.  
\item[(iv)] $a, b \not\in \Q$ but $ab \in \Q$, and $g$ is a real-valued function. 
\end{enumerate}
\end{thm} 

\begin{proof} The proof is divided in a number of cases.

\noindent {\bf (i), (ii), (iv)}  follow from Theorem~\ref{3-2config}.

\vspace{.1in}

\noindent {\bf (iii)} Suppose that  $a, b \not \in \Q.$ Furthermore, assume that $ab\not\in \Q$. Using a metaplectic transformation, we may assume that $\Lambda$ is of the form $\Lambda=\{(0,0), (0,a), (0,-a), (1, b), (1, b)\}$, with $a, b, b/a\not\in \Q$. The rest of the proof is an extension of \cite[Theorem 1.3]{Dem10}. 

We follow the proof of part  {\bf (ii)} of Theorem~\ref{3-2config} and argue by contradiction. In particular, we assume that~\eqref{5lidep},~\eqref{recur5}, and~\eqref{decayinfty} hold for all $x\in I$, where $I\subset \text{supp}(g)\cap [0,1]$ is a set of positive measure. Recall that $P(x)=c_1+c_2e^{2\pi i ax}+c_3e^{-2\pi i ax}$ and $Q(x)=e^{2\pi i (-bx+\theta)}(r_1+r_2e^{2\pi i(2bx+\theta')})$ where $r_1, r_2\in (0, \infty)$, $\theta, \theta' \in [0,1)$, $c_1\in \R$, $c_2, c_3\in \C$ with $c_k\neq 0$ for $k=1,2,3$. 

\noindent We first prove that $\tfrac{|-c_1\pm \sqrt{c_1^2-4c_2c_3}|}{2|c_2|}\neq 1$. Suppose by way of contradiction that $\tfrac{|-c_1\pm \sqrt{c_1^2-4c_2c_3}|}{2|c_2|}= 1$. This implies that $P(x)=0$ has real solutions of the form $$x_k=\omega + \tfrac{k}{a}$$ for some $\omega\in \R$ and $k\in \Z$.

\noindent Next we prove that  $Q$ must also have some real roots. Indeed, assume that $Q(x)\neq 0$ for all $x\in \R$. Since $a\not\in \Q$ we can choose $k\in \Z$ with $x_k>0$ and $\{x_k\} \in I$ (recall that $\{u\}$ is the fractional part of $u$). Note that $g(\{x_k\})\neq 0$. We now use~\eqref{recur5} to get $$0=|P(x_k)||g(x_k)|=|Q(x_k)||g(x_k-1)|$$ Thus $g(x_k-1)=0$. We can continue this iteration to show that $g(x_k-n)=0$ for all $n>0$. Consequently, $g(\{x_k\})=0$ which is a contradiction. Therefore, $Q$ has real roots of the form $$y_n=\omega' +\tfrac{n}{2b}$$  for some $\omega'\in \R$ and $n\in \Z$. 

\noindent Furthermore, the zeros of $P$ and $Q$ must share a $\Z-$orbit. Indeed, if this was not the case, we must have that $x_k-y_n\not\in \Z$ for all $n, k\in \Z$. However, a repeated use of~\eqref{recur5} will lead to the following contradiction. For any  $k\in Z$ we have $0=|P(x_k)||g(x_k)|=|Q(x_k)||g(x_k-1)|$. Since $x_k=x_k-0$ is not a root of $Q$ we see that $g(x_k-1)=0$. Continuing in this fashion we see that $g(x_k-n)=0$ for all $n>0$. Which is a contradiction. In fact, there must exist $n\neq n' \in \Z$ and $m, m'\in \Z$ such that $$x_n-y_m, \, x_{n'}-y_{m'}\in \Z.$$ By taking the difference between these two numbers we see that $$\tfrac{N}{a}+\tfrac{M}{2b}=k$$ for some $N, M, k\in \Z$. Using the fact that $a, b\not\in \Q$ we arrive at the conclusion that all $N, M$ satisfying this equation must be of the form $N=\ell N_0$ and $M=\ell M_0$ for some fixed $N_0, M_0\in \Z\setminus \{0\}$ and arbitrary $\ell \in \Z$. In addition, all $n, m\in \Z$ such that $x_n-y_m\in \Z$ must be of the form 

\begin{equation*}
\left\{ \begin{array} {r@{\quad = \quad}l}
n & n_0+\ell N_0\\
m & m_0+\ell M_0,\end{array}\right.
\end{equation*}

for some fixed $n_0, m_0, N_0, M_0\in \Z$, $N_0, M_0 \neq 0$ and arbitrary $\ell \in \Z$. We also point out that for each $x_n$ there is at most one $y_m$ such that $x_n-y_m\in \Z$.

\noindent  Let $x_\ell= \omega_0+\ell\tfrac{N_0}{a}$ be a zero of $P$ where $\omega_0=\omega+\tfrac{n_0}{a}$, and $y_\ell$ be the zero of $Q$ such that $x_\ell-y_\ell \in \Z$. Note that $y_\ell=  \omega_0'+\ell\tfrac{M_0}{2b}$  where $\omega_0'=\omega'+\tfrac{m_0}{2b}$. Because $\tfrac{N_0}{a}\neq \tfrac{M_0}{2b}$, we can choose $\ell \in \Z$ such that one of the following three alternatives holds:\newline
\noindent $\bullet$ $0<x_\ell<y_\ell$\newline
\noindent $\bullet$ $x_\ell<0<y_\ell$\newline
\noindent $\bullet$ $x_\ell<y_\ell<0$\newline

If we assume that the first alternative holds, by ergodicity, we can choose $\ell\in \Z$ such that $u_\ell=\{x_\ell\}=\{y_\ell\}\in I$. Note that $g(u_\ell)\neq 0$ and using the recursion~\eqref{recur5} and the fact that $Q$ is nonzero on the orbit before $y_\ell$, we see that $g(u_\ell +1)\neq 0$, which implies that $g(u_\ell+2)\neq 0$. We can continue all the way to $g(u_\ell+n)\neq 0$ where $n\in \Z$ is such that $u_\ell +n+1=x_\ell$. Applying~\eqref{recur5} one more time will give $$0=|P(x_\ell)||g(x_\ell)|=|Q(x_\ell)||g(u_\ell+n)|\neq 0$$

It follows that $\inf_{x\in \R}|P(x)|>0$. Similarly, we show that $\inf_{x\in \R}|Q(x)|>0$. Consequently, $\psi(x)=\ln|c_1+c_2e^{2\pi ix}+c_3e^{-2\pi i x}|$ and $\phi(x)=\ln|r_1+r_2e^{2\pi i(2x+\theta')}|$ are well-defined and continuous on $\R$  

Using~\eqref{5lidep} and~\eqref{decayinfty} we see that for each $x, z\in I$

\begin{equation}\label{negativeinf}
\lim_{N\to \infty}\sum_{n=1}^N\phi(bx+bn)-\sum_{n=1}^N\psi(ax+an)=-\infty 
\end{equation}
and 
\begin{equation}\label{positiveinf}
\lim_{N\to \infty}\sum_{n=-N}^{-1}\phi(bz+bn)-\sum_{n=-N}^{-1}\psi(az+an)=\infty 
\end{equation}

We now use the approximation of $a$ by rational and the fact $|\psi'|\gtrsim 1$ to control parts of the above sums. 

Let $p_k, q_k$ relatively prime integers, $q \to \infty$ such that $$|a-\tfrac{p_k}{q_k}|\leq \tfrac{1}{q_k^2}.$$ Furthermore, $$|na-\tfrac{np_k}{q_k}|\leq \tfrac{1}{q_k}, \, -q_k\leq n\leq q_k.$$ By a Riemann sum approximation we see that $$\bigg{|}\sum_{n=1}^{q_k}\psi(ax+an)-q_k\int_0^1\psi\bigg{|}=O(1)$$ and 

$$\bigg{|}\sum_{n=-q_k}^{-1}\psi(ax+an)-q_k\int_0^1\psi\bigg{|}=O(1)$$ for each $x\in [0,1].$ Consequently, for each $y, z\in I$ we have

\begin{equation}\label{sum1}
\bigg{|}\sum_{n=1}^{q_k}\psi(ay+an)-\sum_{n=-q_k}^{-1}\psi(az +an)\bigg{|}=O(1).
\end{equation}

Now using Birkhoff's pointwise ergodic theorem for $1_I$, we can choose $x\in I$, $n'\in N$ such that $z:=\{-x-\tfrac{\theta'}{b}+\tfrac{n'}{2b}\}\in I$. Let $y:=-x-\tfrac{\theta'}{b}+\tfrac{n'}{2b}$ and $m=y-z$. Then $$\sum_{n=-N+m}^{-1+m}\phi(bz+bn)=\sum_{n=1}^N\phi(by-bn)=\sum_{n=1}^N\phi(bx+bn).$$ Observe that for each $N$

\begin{align*}
\sum_{n=-N}^{-1}\phi(bz+bn)&=\sum_{n=-N}^{-N+m-1}\phi(bz+bn)+ \sum_{n=-N+m}^{-1+m}\phi(bz+bn)+\sum_{n=m}^{-1}\phi(bz+bn)\\
&=\sum_{n=1}^N\phi(bx+bn)+ \sum_{n=-N}^{-N+m-1}\phi(bz+bn)+\sum_{n=m}^{-1}\phi(bz+bn)
\end{align*}

Consequently, for each $N$ 

\begin{align}\label{sum2}
\bigg{|}\sum_{n=-N}^{-1}\phi(bz+bn)-\sum_{n=1}^N\phi(bx+bn)\bigg{|}&=\bigg{|}\sum_{n=-N}^{-N+m-1}\phi(bz+bn)+ \sum_{n=m}^{-1}\phi(bz+bn)\bigg{|}\\ \notag
&=O(1)
\end{align} where we bound the last sum by a constant that depends only on $m, z,$ and $b$. However, \eqref{negativeinf}--\eqref{sum2} cannot simultaneously hold. This completes the proof.

Next we suppose that $a, b\not\in \Q$ but $ab\in \Q$. 

%
%
%
%

\end{proof}

 We note that as observed in \cite[Theorem 1.3]{Dem10}, rather than assuming that $g\in \S(\R)$ we could assume that $g\in L^2(\R)$ is continuous and is such that $\lim_{|n|\to \infty\, n\in \Z}|g(x-n)|=0$ for all $x\in [0,1]$. 

\begin{rem} Suppose that $g\in \S(\R)$ is real-valued. In addition to the cases covered by Corollary~\ref{4points-hrt}, Theorem~\ref{smooth-3-2config} can be used to settle Conjecture~\ref{hrt-s} when  $a, b, ab\not \in \Q$.

\end{rem}

%

\begin{rem}
We summarize what is known about the HRT for $4$ points. In $\R^2$, a set  $\Lambda$ consisting of four distinct points in $\R^2$ can be such that:
\begin{enumerate}
\item[(1)] The four points are  collinear, in which case their convex hull is a line segment,
\item[(2)] The four points form a $(1,3)$ configuration,  in which case their convex hull is a triangle, 
\item[(3)] The four points  form a $(2,2)$ configuration,  in which case their convex hull is a trapezoid, or
\item[(4)] The four points are in none of the previous three categories,  in which case their convex hull is a general quadrilateral. 
\end{enumerate}

\noindent $\bullet$ In the first case Conjecture~\ref{hrt} holds for any $g \in L^2(\R)$ \cite{HRT96}.\newline
\noindent $\bullet$ In the second case Conjecture~\ref{hrt-s} holds for any $g \in \S(\R)$ \cite{Dem10}. However, if $g \in L^2$,  Conjecture~\ref{hrt} is true  when the three collinear points are also equispaced \cite{HRT96}. More generally, \cite{Dem10} has condition under which the conjecture remains true, and,  in fact, Conjecture~\ref{hrt} holds for almost all $(1,3)$ configurations. The results of this paper allow us to conclude that when $g\in L^2$ is also real-valued then Conjecture~\ref{hrt} for all $(1,3)$ configurations.\newline
\noindent $\bullet$ For the third case, Conjecture~\ref{hrt} holds for any $g \in L^2(\R)$ \cite{DemZah12}.\newline
\noindent $\bullet$ In the last case,  to the best of our knowledge both Conjecture~\ref{hrt} and Conjecture~\ref{hrt-s} remain open. However,  when $g(x)=e^{-|x|^{\epsilon}}$ with $\epsilon>0$ Conjecture~\ref{hrt} holds for any set of $4$ points \cite{BeBo13}. In fact, when $g$ decays faster than any exponential the HRT conjecture has been established not only in dimension one, but also in higher dimensions \cite{BeBo13, BowSpee13, BowSpee16}. In this paper we showed that when $g\in L^2(\R)$ is real-valued  Conjecture~\ref{hrt} holds for a family of four distinct points.  

But is still unclear whether the HRT holds for any $4$ points and every $g\in L^2(\R)$.

\end{rem}

\section*{Acknowledgment} 
The author thanks C.~Heil for introducing him to this fascinating and addictive problem, and for invaluable comments and remarks on earlier versions of this paper. He also thanks R.~Balan, J.~J.~Benedetto, and D.~Speegle for helpful discussions over the years about various versions of the results presented here. He acknowledges C.~Clark's help in generating the pictures included in the paper. Finally, he thanks W.~Liu for helpful discussions, and the anonymous referees for their useful and insightful comments and remarks. 

\bibliographystyle{amsplain}
\bibliography{MS17075_Bib}

\end{document}